\documentclass[11pt]{amsart}

\usepackage{amscd}
\usepackage{easybmat}
\usepackage{mathrsfs}
\usepackage{amsfonts}
\usepackage{color}
\usepackage{pifont}
\usepackage{upgreek}
\usepackage{bm}
\usepackage{hyperref}
\usepackage{shorttoc}
\usepackage{amsmath,amstext,amsthm,a4,amssymb,amscd}
\usepackage[mathscr]{eucal}
\usepackage{mathrsfs}
\usepackage{epsf}
\numberwithin{equation}{section}
\newtheorem{theo}{Theorem} 

\newtheorem{lem}{Lemma}
\newtheorem{mcor}{Corollary}
\newtheorem{remark}{Remark}
\newtheorem{prop}{Proposition}
\newtheorem{definition}{Definition}

\newcommand*{\tr}{\mathrm{tr}}

\newcommand*{\D}[1]{\ensuremath{\nabla^{#1}}}

\usepackage{fancyhdr}
\begin{document}
\title[On equivalence theorems of Minkowski spaces and an application]{On equivalence theorems of Minkowski spaces and an application in Finsler geometry}

\author[Ming Li]{Ming Li$^{\dag}$}

\address{Ming Li: School of Science,
 Chongqing University of Technology,
Chongqing 400054, People's Republic of China }

\email{mingli@cqut.edu.cn}

\thanks{$^{\dag}$~Partially supported by NSFC (grant No. 11501067, 11571184) and the Marie Cuire IRSES project
(grant agreement PIRSES-GA-2012-317721-LIE-DIFF-GEOM)}


\date{}  
\maketitle

\begin{abstract}
In this paper, we first establish an equivalence theorem of Minkowski spaces by using results in centro-affine differential geometry.
As an application in Finsler geometry, we gives some new characterizations of Berwald spaces.

\begin{flushleft}

\end{flushleft}

\end{abstract}


\section*{Introduction}

In this paper, a Minkowski space means an $n$ dimensional real vector space $V$ with a smooth strongly convex Minkowski norm $\mathbf{F}$.
Two Minkowski spaces $(V_1,\mathbf{F}_1)$ and  $(V_2,\mathbf{F}_2)$ are equivalent, if there exists a nondegenerate linear homomorphism
$L:V_1\rightarrow V_2$, such that $\mathbf{F}_1=\mathbf{F}_2\circ L$.
For example, all $n$ dimensional vector spaces with Euclidean norms, which are such norms induced by inner products, are equivalent to $(\mathbb{R}^n,\|\cdot\|)$,
where $\|y\|=\sqrt{\sum (y^i)^2}$ for $y=(y^1,\ldots,y^n)\in \mathbb{R}^n$.

On a Minkowski space $(V,\mathbf{F})$, two basic tensors deduced from $\mathbf{F}$ are the fundamental form $\hat{g}$ and
the Cartan tensor $\hat{A}$, respectively. The trace of $\hat{A}$ is the Cartan form $\eta$. In Section 2, we will use these geometric invariants to
establish the following equivalence theorem of Minkowski spaces.
\begin{theo}\label{theo equivalence}
Let $(V_1,\mathbf{F}_1)$ and $(V_2,\mathbf{F}_2)$ be two Minkowski spaces of dimension $n>2$, respectively.
Let $f:V_1\rightarrow V_2$ be a norm preserving map which is a diffeomorphism on $(V_1)_0:=V_1\setminus\{0\}$, and satisfies
\begin{align*}
f(tv)=tf(v), \quad \forall v\in V_1,~\forall t>0.
\end{align*}
Then $\hat{g}_1=f^*\hat{g}_2$ and $\eta_1=f^*\eta_2$ if and only if there exists a nondegenerate linear homomorphism $L\in{\rm Hom}(V_1,V_2)$, such that $f=L$ and
$$\mathbf{F}_1=\mathbf{F}_2\circ L.$$
\end{theo}
This theorem is a generalization of the equivalence theorem for Minkowski plane (cf. Proposition 4.4.1 in \cite{BaoChernShen}, p. 90).
Our strategy is to establish a canonical correspondence between the differential geometry of a Minkowski space
and the centroaffine differential geometry of its indicatrix and then apply the uniqueness theorems for hyperovaloids of Schneider in \cite{Schneider1} in affine geometry.
This idea has been noticed and used by many authors (\cite{Paiva,Bla,Bryant,H,Laugwitz0,Laugwitz1,Laugwitz2,MoHuang}). The mentioned correspondence was first introduced by Laugwitz in \cite{Laugwitz0,Laugwitz1,Laugwitz2}. However, we would like to gives the details in this paper for modern readers in the English world.

We then investigate Finsler geometry.
Let $M$ be an $n$-dimensional smooth manifold and $\pi:TM\to M$
the tangent bundle of $M$. Set $TM_0=TM\setminus0$, where $0$ denotes the zero section of $TM$.
A Finsler structure on $M$ is a continue function
$\mathbf{F}:TM\rightarrow\mathbb{R}$, which is smooth on $TM_0$, such that the restriction $\mathbf{F}_{T_pM}$ is a Minkowski norm for each $p\in M$. A manifold $M$ with a Finsler structure $\mathbf{F}$ is called a Finsler manifold, and denoted by
$(M,\mathbf{F})$. The Finsler structure $\mathbf{F}$ induces a canonical spray on $TM_0$.
This gives rise a natural horizontal subbundle of $T(TM_0)$ or a nonlinear connection of $M$.

On one hand, the natural splitting of $T(TM_0)$ allows us to introduce three vertical tensors $\hat{g}$, $\hat{A}$ and $\eta$ on $TM_0$.
For any $p\in M$, let $i_p:T_pM\hookrightarrow TM$ be the natural embedding mapping. The restrictions $i_p^*\hat{g}$, $i_p^*\hat{A}$ and $i_p^*\eta$ on $T_pM\setminus\{0\}$ are the fundamental tensor, Cartan tensor and Cartan form of $(T_pM, \mathbf{F}_{T_pM})$, respectively. These tensors are geometric invariants which are independent to any linear connection of $(M,\mathbf{F})$.

On another hand, the nonlinear connection gives rise the (nonlinear) parallel transport on the Finsler manifold $(M,\mathbf{F})$.
Let $\sigma:[a,b]\rightarrow M$ be a piecewise smooth curve form $\sigma(a)=p$ to $\sigma(b)=q$,
one can define the parallel transport
\begin{equation*}
  P_{\sigma,t}:T_pM\rightarrow T_{\sigma(t)}M, \quad {\rm for}~t\in[a,b].
\end{equation*}
It is well known that the restricted map $P_{\sigma,t}:T_pM\setminus\{0\}\rightarrow
T_{\sigma(t)}M\setminus\{0\}$ is a norm preserving diffeomorphism,
and satisfies
\begin{align*}
P_{\sigma,t}(\lambda y)=\lambda P_{\sigma,t}(y), \quad \forall ~\lambda>0, \forall~y\in T_pM\setminus\{0\}.
\end{align*}
Parallel transports are usually nonlinear. In fact, $(M,\mathbf{F})$ is a Berwald manifold if and only if all the parallel transports are linear (cf. \cite{Bao,BaoChernShen,ChernShen}). This characterization of Berwald manifold initiates by Ichijy\={o} in \cite{Ichijyo76}.
Combining this fact with Theorem 1, we immediately get the following theorem.
\begin{theo}\label{2}
  Let $(M,\mathbf{F})$ be a Finsler manifold of dimension $n>2$. $(M,\mathbf{F})$ is a Berwald manifold
  if and only if the vertical tensors $\hat{g}$ and $\eta$ are preserved by all parallel transports.
\end{theo}
It is proved in \cite{Aikou10,Bao,ChernShen,Ichijyo78} that $\hat{g}$ is preserved by all parallel transports if and only if
$(M,\mathbf{F})$ is a Landsberg manifold, which has vanishing Landsberg curvature (or referred as the second Cartan tensor in \cite{SV} etc.).

Motivated by Theorem \ref{2} and our previous work \cite{FL}, we investigate $\eta$ and get some interesting properties. Then we give the following characterization theorem of Berwald manifolds.
\begin{theo}\label{3}
Let $(M,\mathbf{F})$ be a Finsler manifold of dimension $n>2$, then the following statements are equivalent:
\begin{enumerate}
\item[(1)] $M$ is a Landsberg manifold with vanishing $S$-curvature;

\item[(2)] $M$ is a Landsberg manifold with $\eta$ is exact;

\item[(3)] $M$ is a Landsberg manifold with $d\eta=0$;

\item[(4)]  $M$ is a Berwald manifold.
\end{enumerate}
\end{theo}

As a corollary, we have the following theorem which gives a positive answer of a problem asked by Zhongmin Shen (cf. \cite{Shen}, p. 322).
\begin{theo}\label{theo L=0 and trB=0}
A Landsberg manifold $(M,\mathbf{F})$ with vanishing mean Berwald curvature must be Berwaldian.
\end{theo}


This paper contains 4 sections. In Section 1, we give a brief review of some concepts and results in centroaffine differential
geometry of hypersurfaces in real vector spaces. In  Section 2, we will establish the correspondence between Minkowski spaces and hyperovaloids with origin in its interior. After that, we will proof Theorem \ref{theo equivalence}.
In Section 3, we will study the local geometry of some geometric invariants on Finsler manifolds by Chern connection and its curvature.
Especially, we investigate the Cartan-type form and find some relations with other known geometric invariants.
In the last section, we will use parallel transport to describe some special Finsler manifolds. Then we will prove Theorem \ref{2}--\ref{theo L=0 and trB=0}.

In this paper, lower case Latin indices will run from 1 to $n$ and
lower case Greek indices will run from 1 to $n-1$. We also adopt the
summation convention of Einstein. We will assume in this paper that $n>2$.

\

{\bf Acknowledgements.} The author would like to thank Professor
Huitao Feng for his consistent support and encouragement. The author
would like to express his deep appreciation to Professor Guofang
Wang for his warm hospitality and discussions on Mathematics, while the author stayed in
Mathematical Institute Albert Ludwigs University Freiburg in 2014. The author would like to thank referees for helpful
comments and suggestions.

\section{Review of centroaffine differential geometry of hypersurfaces}
In this section, we would like to review the fundamental equations and some results of an affine hypersurface with centroaffine normalization.
One refers to \cite{LSZ,NS,SSV} for details.

\subsection{Centroaffine normalization of a nondegenerate hypersurface}

\

Let $V$ be a real vector space of dimensional $n$ with a chosen orientation.
Let $V^*$ be the dual space of $V$, and $\langle~,~\rangle:V\times V^*\rightarrow \mathbb{R}$ the canonical pairing.
$V$ has a smooth manifold structure and a flat affine connection $\bar{\nabla}$.


Let $x:M\rightarrow V$ be an immersed connected oriented smooth manifold $M$ of dimension $n-1$.
Then for each point of $p\in M$, $dx(T_pM)$ is an $n-1$ dimensional subspace of $T_{x(p)}V$,
and defines an one dimensional subspace $C_pM=\{v^*_p\in V^*|\ker v^*_p=dx(T_pM)\}\subset T^*_{x(p)}V$.
The trivial line bundle $CM=\bigcup_pC_pM$ is called the conormal line bundle of $x$. 

Let $Y$ be a nowhere vanishing section of $CM$.
If ${\rm rank}(dY,Y)=n$, then $x$ is called a nondegenerate hypesurface. The nondegenerate property is independent to the choice of the conormal field $Y$.
In this paper, we will only discuss nondegenerate hypersurfaces.
Let $y:M\rightarrow V$ be the special vector filed $y(p)=-x(p)$, $\forall p\in M$.
If $\langle Y, y\rangle=1$, then the pair $\{Y,y\}$ is called the centroaffine normalization of $x$.
We state the structure equations of $x(M)$ with respect to the centroaffine normalization $\{Y,y\}$ as following,
\begin{align*}
\bar{\nabla}_vy=dy(v)=-dx(v),
\end{align*}
\begin{align}\label{Gauss eq for x}
\bar{\nabla}_vdx(w)=dx(\nabla_vw)+h(v,w)y,
\end{align}
\begin{align*}
\bar{\nabla}^*_vdY(w)=dY(\nabla^*_vw)-h(v,w)Y,
\end{align*}
where $h$ is a nondegenerate symmetric $(0,2)$-tensor and called the induced affine metric,
$\nabla$ and $\nabla^*$ are torsion free affine connections.
These geometric quantities satisfy
\begin{align}
dh(v_1,v_2)=h(\nabla v_1,v_2)+h(v_1,\nabla^*v_2). \label{conjugate connections}
\end{align}
The triple $\{\nabla,h,\nabla^*\}$  are called conjugate connections.
For any triple of conjugate connections $\{\nabla,h,\nabla^*\}$, one can define
$C=\frac{1}{2}(\nabla-\nabla^*)\in\Omega^1(M,{\rm End}(TM))$.
By (\ref{conjugate connections}), the $(0,3)$-tensor $\hat{C}:=h\circ C$ is totally symmetric and called the cubic form of $\{\nabla,h,\nabla^*\}$.
One can prove that
\begin{align}\label{Chat}
\hat{C}=-\frac{1}{2}\nabla h.
\end{align}
For $\{\nabla,h,\nabla^*\}$, the Tchebychev form $\hat{T}$ is defined as the normalized trace of $C$,
\begin{align*}
\hat{T}=\frac{1}{n-1}{\rm tr}C.
\end{align*}
The Tchebychev field $T$ is the dual of $\hat{T}$ with respect to $h$.

Let $\omega(h)$ be the Riemannian volume of $h$ on $x(M)$ and $\omega$ the induced volume form on $x(M)$ from the orientation of $V$.
Then it is proved that
\begin{align}\label{That2}
\hat{T}=\frac{1}{n-1}d\log\left|\frac{\omega}{\omega(h)}\right|.
\end{align}

\subsection{Uniqueness theorem for hyperovaloids}

\

The local existence and uniqueness theorem for nondegenerate hypersurfaces is classical (cf. \cite{SSV}, 6.3.3). However, there is
a remarkable global uniqueness theorems for hyperovaloids.
We would like to review some facts about hyperovaloid.
\begin{lem}
Let $M$ be an $n-1$ dimensional connected closed smooth manifold.
Let $x:M\rightarrow V$ be a smooth immersion in an $n$ dimensional real vector space $V$.
If the centroaffine normalization of $x(M)$ is nondegenerate, then $x(M)$ is a hyperovaloid.
\end{lem}

\begin{lem}[\cite{NS}, Prop.7.3]
Assume that $x:M\rightarrow V$ is an $n$ dimensional hyperovaloid with respect to the centroaffine normalization. Then

{\upshape(i)}\quad $M$ is diffeomorphic to the $n-1$ dimensional standard sphere $\mathbb{S}^{n-1}$;

{\upshape(ii)}\quad $x$ is an imbedding;

{\upshape(iii)}\quad $x(M)$ is the boundary of a strongly convex body, in which the origin is contained.
\end{lem}

\begin{remark}
In \cite{NS}, two terms about convexity are used. They are ``locally strictly convex'' and ``globally strictly convex'' respectively.
But in \cite{BaoChernShen}, the convexity ``locally strictly convex'' is named as ``strongly convex'', and ``globally strictly convex'' as ``strictly convex''.
It is proved that strongly convexity implies strictly convexity in \cite{BaoChernShen}. Here we follow \cite{BaoChernShen} terminologically.
\end{remark}

Now we are going to state a remarkable uniqueness theorem for hyperovaloids with centroaffine normalizations, which is due to R. Schneider\cite{Schneider1}.

\begin{lem}[\cite{Schneider1}, Satz 4.2]\label{global unique c}
Let $x_i:M\rightarrow V_i$ be two nondegenerate immersions of a connected closed $n-1$ dimensional manifold $M$ in
an $n>2$ dimensional real vector space $V_i$ with the
centroaffine normalizations $\{Y_i,y_i\}$, $i=1,2$. Let
$\hat{T}_i$ be the Tchebychev forms of
$x_i$, $i=1,2$, respectively.

Then $h_1=h_2$ and $\hat{T}_1=\hat{T}_2$ if and only if there exists a nondegenerate linear homomorphism $L\in{\rm Hom}(V_1,V_2)$, such that
$$x_2=L\circ x_1.$$
\end{lem}

\section{Equivalence theorems of Minkowski spaces}

 For convenience, we would like to reestablish the canonical correspondence between the differential geometry of a Minkowski space and the centroaffine differential geometry of its indicatrix. Then we will give a proof of Theorem \ref{theo equivalence} by applying Lemma \ref{global unique c}.

\subsection{Minkowski spaces and their indicatrices}

\

Let $V$ be an $n$ dimensional real vector space with a given orientation.
Let $\{\mathbf{b}_1,\ldots,\mathbf{b}_n\}$ be an oriented basis.
Then the mapping $\phi:V\rightarrow \mathbb{R}^n$ defined by
$$\phi(y)=(y^1,\ldots,y^n),\quad \forall y=y^i\mathbf{b}_i\in V,$$
gives the standard smooth structure of $V$. Let $\{\mathbf{b}^*_1,\ldots,\mathbf{b}^*_n\}$ be the dual basis.
The orientation of $V$ is then the $n$ form $\omega=\mathbf{b}^*_1\wedge\cdots\wedge\mathbf{b}^*_n$, which gives a volume form on $V$.

For any function $\mathbf{f}:V\rightarrow \mathbb{R}$, we will denote $f=\mathbf{f}\circ\phi^{-1}$.
$\mathbf{f}$ is said to be differentiable on $V$, if $f$ is differentiable as a function on $\mathbb{R}^n$.

In the following, we review some basic concepts.

\begin{definition}
Let $V$ be an $n$ dimensional real vector space with the standard smooth structure $(y^i;\phi)$.
Let $\mathbf{F}:V\rightarrow[0,+\infty)$ be a function, such that
\begin{enumerate}
\item[(i)] $\mathbf{F}$ is continuous on $V$, and smooth on $V_0:=V\setminus\{0\}$;
\item[(ii)] $\mathbf{F}(\lambda v)=\lambda \mathbf{F}(v)$, \quad $\forall v\in V,\quad \forall \lambda\in \mathbb{R}^+$;
\item[(iii)] $\mathbf{F}$ is strongly convex, i.e., the symmetric tensor $\bar{g}:=\bar{\nabla}d\left[\frac{1}{2}\mathbf{F}^2\right]$ is positive everywhere.
\end{enumerate}
Then $\mathbf{F}$ is called a Minkowski norm of $V$. $(V,\mathbf{F})$ is called a Minkowski space.

\end{definition}

By definition, $(V_0,\bar{g})$ is a Riemannian manifold related to a Minkowski space $(V,\mathbf{F})$.
The Riemannian geometry of $(V_0,\bar{g})$ is fundamental and has many applications in Finsler geometry.
But the following concept concentrates almost all geometric information of a Minkowski space.
\begin{definition}
Let $(V,\mathbf{F})$ be a Minkowski space of dimension $n$. The set
\begin{align}
\mathbf{I}_{\mathbf{F}}:=\{v\in V|\mathbf{F}(v)=1\}
\end{align}
is called the indicatrix of the Minkowski space $(V,\mathbf{F})$.
\end{definition}
The indicatrix $\mathbf{I_F}$ of a Minkowski space $(V,\mathbf{F})$ has been studied as a submanifold of $(V_0,\bar{g})$.
One refers to \cite{BaoChernShen} for details.
In the following, we will study the centroaffine geometry of $\mathbf{I_F}$.
It will be proved that these two kinds of geometry of $\mathbf{I_F}$ are in fact the same.
But the centroaffine differential geometry are more suitable for the discussion of equivalence problem.

\begin{definition}
Let $(V,\mathbf{F})$ and $(\tilde{V},\tilde{\mathbf{F}})$ be two Minkowski spaces of dimension $n$.
If there exists a homomorphism $L\in {\rm Hom}(\tilde{V},V)$, such that
$$\tilde{\mathbf{F}}=\mathbf{F}\circ L,$$
then $(V,\mathbf{F})$ and $(\tilde{V},\tilde{\mathbf{F}})$ are said be equivalent. We will denote
$(V,\mathbf{F})\sim (\tilde{V},\tilde{\mathbf{F}})$, if $(V,\mathbf{F})$ and $(\tilde{V},\tilde{\mathbf{F}})$ are equivalent.
\end{definition}
It is clear that $``\sim"$ is an equivalent relation on the set of $n$ dimensional Minkowski spaces.

\begin{lem}\label{lemma 2.1}
Let $(V,\mathbf{F})$ and $(\tilde{V},\tilde{\mathbf{F}})$ be two Minkowski spaces of dimension $n$.
Let $\mathbf{I}_{\mathbf{F}}$ and $\mathbf{I}_{\tilde{\mathbf{F}}}$ be their indicatrices respectively.
Then $(V,\mathbf{F})\sim (\tilde{V},\tilde{\mathbf{F}})$ if and only if
$$\mathbf{I}_{\mathbf{F}}=L(\mathbf{I}_{\tilde{\mathbf{F}}}),$$
holds for some $L\in {\rm Hom}(\tilde{V},V)$.
\end{lem}
\begin{proof}
We first assume that $(V,\mathbf{F})\sim (\tilde{V},\tilde{\mathbf{F}})$.
So $\tilde{\mathbf{F}}=\mathbf{F}\circ L$ holds for some $L\in {\rm Hom}(\tilde{V},V)$.
For any $\tilde{v}\in \mathbf{I}_{\tilde{\mathbf{F}}}$, we have
\begin{align*}
\mathbf{F}(L(\tilde{v}))=(\mathbf{F}\circ L)(\tilde{v})=\tilde{\mathbf{F}}(\tilde{v})=1.
\end{align*}
So $L(\tilde{v})\in \mathbf{I}_{\mathbf{F}}$.

For any $v\in \mathbf{I}_{\mathbf{F}}$, one chooses that $\tilde{v}=L^{-1}(v)$.
Since
\begin{align*}
\tilde{\mathbf{F}}(\tilde{v})=(\mathbf{F}\circ L)(L^{-1}(v))=\mathbf{F}(v)=1,
\end{align*}
then $\tilde{v}\in\tilde{\mathbf{F}}(\tilde{v})$ and $v=L(\tilde{v})$. So we have proved that $\mathbf{I}_{\mathbf{F}}=L(\mathbf{I}_{\tilde{\mathbf{F}}})$.

Now, we assume that $\mathbf{I}_{\mathbf{F}}=L(\mathbf{I}_{\tilde{\mathbf{F}}})$ for some $L\in {\rm Hom}(\tilde{V},V)$.
For any $\tilde{v}\in\tilde{V_0}$, it is clear that ${\lambda}^{-1}\tilde{v}\in \mathbf{I}_{\tilde{\mathbf{F}}}$, where $\lambda=\tilde{\mathbf{F}}(\tilde{v})$.
Since $L\left({\lambda}^{-1}\tilde{v}\right)\in \mathbf{I}_{\mathbf{F}}$, we have
\begin{align*}
(\mathbf{F}\circ L)(\tilde{v})=\mathbf{F}(L(\tilde{v}))=\lambda\mathbf{F}\left(L\left({\lambda}^{-1}\tilde{v}\right)\right)=\lambda.
\end{align*}
It implies that $\tilde{\mathbf{F}}=\mathbf{F}\circ L$ for $L\in {\rm Hom}(\tilde{V},V)$.
\end{proof}

Before the study of the centroaffine differential geometry of the indicatrices,
We would like to review the definitions of three basic tensor fields on $(V_0,\bar{g})$.

\begin{definition}
Let $(V,\mathbf{F})$ be a Minkowski space of dimension $n$. We define a conformal metric of $\bar{g}$ as
\begin{align*}
\hat{g}=\frac{1}{\mathbf{F}^{2}}\bar{g}.
\end{align*}
The Cartan tensor is defined by
\begin{align*}
\hat{A}=\frac{1}{2\mathbf{F}^{2}}\bar{\nabla}\bar{g}.
\end{align*}
Let $|\mathcal{G}|=\left|\frac{\omega(\bar{g})}{\omega}\right|$ be the Radon-Nikodym derivative of the Riemannian measure of $\bar{g}$ with respect to the measure $\omega=\mathbf{b}^*_1\wedge\cdots\wedge\mathbf{b}^*_n$ induced by the orientation of $V$.
The Cartan form is defined by
\begin{align*}
\eta=d\log|\mathcal{G}|.
\end{align*}
\end{definition}
\begin{remark}
One notes that the tensor fields $\hat{g}$, $\hat{A}$ and $\eta$ are all positive homogeneous of degree $0$.
And $i_{\mathbf{x}}\hat{A}=0$ and  $i_{\mathbf{x}}\eta=0$, where $\mathbf{x}$ is the position vector field on $V$.
\end{remark}

\begin{lem}\label{lemma 2.2}
The tensors $\hat{g}$, $\hat{A}$ and $\eta$ are relative invariants with respect to the equivalence relation of Minkowski spaces.
\end{lem}
\begin{proof}
Let  $(V,\mathbf{F})$ be a Minkowski space. Let $\{\mathbf{b}_1,\ldots,\mathbf{b}_n\}$ be an oriented basis of $V$.
Let $\phi$ be the induced coordinate map from $V$ to $\mathbb{R}^n$. Then we have $d\mathbf{F}=\phi^*dF,$
and
\begin{align*}
\bar{g}=\phi^*((F_{y^i}F_{y^j}+FF_{y^iy^j})dy^i\otimes dy^j),
\end{align*}
where $F=\mathbf{F}\circ\phi^{-1}$.

Let $(\tilde{V},\tilde{\mathbf{F}})$ be a Minkowski space which is equivalent to $(V,\mathbf{F})$.
Then there exists $L\in {\rm Hom}(\tilde{V},V)$ such that $\tilde{\mathbf{F}}=\mathbf{F}\circ L$.
Let $\tilde{\mathbf{b}}_i=L^{-1}(\mathbf{b}_i)$, $i=1,\ldots,n$.
Let $\tilde{\phi}$ be the induced coordinate map from $\tilde{V}$ to $\mathbb{R}^n$. Then we have $\tilde{\phi}=\phi\circ L$.
It follows that $\tilde{F}=\tilde{\mathbf{F}}\circ\tilde{\phi}^{-1}=(\mathbf{F}\circ L)\circ(\phi\circ L)^{-1}=F$. So we obtain
\begin{align*}
\tilde{\bar{g}}&=\tilde{\phi}^*((\tilde{F}_{y^i}\tilde{F}_{y^j}+\tilde{F}\tilde{F}_{y^iy^j})dy^i\otimes dy^j)\\
&=L^{*}\circ\phi^*((F_{y^i}F_{y^j}+FF_{y^iy^j})dy^i\otimes dy^j)\\
&=L^{*}\bar{g}.
\end{align*}
By the same way, we can prove that $\tilde{\hat{A}}=L^{*}\hat{A}$ and $\tilde{\eta}=L^{*}\eta$.

\end{proof}

We are going to discuss the centroaffine differential geometry of the indicatrix of a given Minkowski space.

\begin{lem}\label{lemma 2.3}
Let $\mathbf{I}_\mathbf{F}$ be the indicatrix of a Minkowski space $(V,\mathbf{F})$.
Then the identity map $i:\mathbf{I}_\mathbf{F}\rightarrow V$ is an imbedding.
For each point $v\in \mathbf{I}_\mathbf{F}$, if we choose that $y=-v\in T_vV$, $Y=-d\mathbf{F}\in T^*_vV$, then the pair $\{Y,y\}$ gives the centroaffine normalization of $\mathbf{I}_\mathbf{F}$.

\end{lem}

\begin{proof}
Since the Minkowski norm $\mathbf{F}:V_0\rightarrow \mathbb{R}^+$ is a smooth mapping, and
\begin{align*}
\mathbf{F}(\lambda v)=\lambda \mathbf{F}(v),\quad \forall~ \lambda>0,
\end{align*}
then
\begin{align*}
\mathbf{F}_{*v}\left(\frac{v}{F(v)}\right)=\frac{\partial}{\partial
t}\in T_{\mathbf{F}(v)}\mathbb{R}^+,\quad v\in T_vV.
\end{align*}
So $\mathbf{F}$ is a submersion. By the inverse image of a regular value theorem, $\mathbf{I}_\mathbf{F}=\mathbf{F}^{-1}(1)$ is an imbedding submanifold of $V_0$.

For each smooth curve $\gamma:(-\epsilon,\epsilon)\rightarrow \mathbf{I}_\mathbf{F}$ such that $\gamma(0)=v\in\mathbf{I_F}$, it is clear $\mathbf{F}(\gamma(t))=1$.
Then
\begin{align*}
0=\left.\frac{d}{dt}\right|_{t=0}\mathbf{F}(\gamma(t))=\langle d\mathbf{F},\dot{\gamma}(0)\rangle,
\end{align*}
and $-d\mathbf{F}$ is a conormal field on $\mathbf{I}_\mathbf{F}$.

By the Euler theorem for homogeneous function,
\begin{align*}
\langle -d\mathbf{F}, -v\rangle|_{v}=\left.\left\langle F_{y_i}dy^i,y^j\frac{\partial}{\partial y^j}\right\rangle\right|_{\phi(v)}=F_{y_i}y^i|_{\phi(v)}=1,
\end{align*}
where $F=\mathbf{F}\circ\phi^{-1}$.
So $\{Y=-d\mathbf{F},y=-v\}$ is the centroaffine normalization.
\end{proof}

Let $\{Y,y\}$ be the centroaffine normalization of $\mathbf{I}_\mathbf{F}$ determined in Lemma \ref{lemma 2.3}.
Then we have the induced affine metric $h$ on $\mathbf{I}_\mathbf{F}$ by the Gauss equation (\ref{Gauss eq for x}).
\begin{lem}\label{lemma 2.4}
The induced Riemannian metric $h$ is given by
\begin{align}
h=i^*\hat{g}=i^*\mathbf{h},
\end{align}
where $i:\mathbf{I_F}\rightarrow V$ is the identity map, and $\mathbf{h}:=\mathbf{F}\bar{\nabla}d\mathbf{F}$ is the angular metric.
As a consequence, $h$ is positive definite and $\mathbf{I}_\mathbf{F}$ is a hyperovaloid.
\end{lem}

\begin{proof}
The equation (\ref{Gauss eq for x}) reads
\begin{align*}
\bar{\nabla}_ui_*w=i_*(\nabla_uw)+h(u,w)y,
\end{align*}
where $u,w\in\Gamma(T\mathbf{I_F})$ are smooth vector fields on $\mathbf{I_F}$, and $\nabla$ is the induced affine connection.

Since $\hat{g}=\mathbf{F}^{-2}(d\mathbf{F}\otimes d\mathbf{F}+\mathbf{F}\bar{\nabla}d\mathbf{F})$, we have
\begin{align*}
i^*\hat{g}(u,w)&=\mathbf{F}^{-2}(d\mathbf{F}\otimes d\mathbf{F}+\mathbf{F}\bar{\nabla}d\mathbf{F})(i_*u,i_*w)\\
&=\bar{\nabla}d\mathbf{F}(i_*u,i_*w)\\
&=\langle\bar{\nabla}_ud\mathbf{F},i_*w\rangle\\
&=\langle-d\mathbf{F},\bar{\nabla}_ui_*w\rangle\\
&=h(u,w).
\end{align*}
So $h$ is positive definite everywhere and $\mathbf{I_F}$ is a hyperovaloid.
\end{proof}

\begin{lem}\label{lemma 2.5}
The cubic form $\hat{C}$ is given by
\begin{align}\label{Chat=Ahat}
\hat{C}=-i^*\hat{A},
\end{align}
and the Tchebychev form is
\begin{align}\label{That=eta}
\hat{T}=-\frac{1}{n-1}i^*\eta
\end{align}
\end{lem}

\begin{proof}
Let $u,w,z\in\Gamma(T\mathbf{I}_F)$ be smooth vector fields on $\mathbf{I_F}$.
Since that
\begin{align*}
A&=\frac{1}{2\mathbf{F}^{2}}\bar{\nabla}(d\mathbf{F}\otimes d\mathbf{F}+\mathbf{F}\bar{\nabla}d\mathbf{F})\\
&=\frac{1}{2\mathbf{F}^{2}}(\bar{\nabla}d\mathbf{F}\otimes d\mathbf{F}+2d\mathbf{F}\otimes \bar{\nabla}d\mathbf{F}+\mathbf{F}\bar{\nabla}\bar{\nabla}d\mathbf{F}),
\end{align*}
and
\begin{align*}
\bar{\nabla}d\mathbf{F}(u,y)&=\langle\bar{\nabla}_ud\mathbf{F},y\rangle\\
&=\langle-d\mathbf{F},\bar{\nabla}_uy\rangle\\
&=\langle-d\mathbf{F},-i_*u\rangle\\
&=0.
\end{align*}
By (\ref{Chat}), we have
\begin{align*}
2i^*A(z,u,w)&=\frac{1}{\mathbf{F}^{2}}(\bar{\nabla}d\mathbf{F}\otimes d\mathbf{F}+2d\mathbf{F}\otimes \bar{\nabla}d\mathbf{F}+\mathbf{F}\bar{\nabla}\bar{\nabla}d\mathbf{F})(i_*z,i_*u,i_*w)\\
&=\bar{\nabla}\bar{\nabla}d\mathbf{F}(i_*z,i_*u,i_*w)\\
&=(\bar{\nabla}_z(\bar{\nabla}d\mathbf{F}))(i_*u,i_*w)\\
&=z(\bar{\nabla}d\mathbf{F}(i_*u,i_*w))-\bar{\nabla}d\mathbf{F}(\bar{\nabla}_zi_*u,i_*w)-\bar{\nabla}d\mathbf{F}(i_*u,\bar{\nabla}_zi_*w)\\
&=z(\bar{\nabla}d\mathbf{F}(i_*u,i_*w))-\bar{\nabla}d\mathbf{F}(i_*(\nabla_zu),i_*w)-\bar{\nabla}d\mathbf{F}(i_*u,i_*(\nabla_zu))\\
&=z(i^*(\bar{\nabla}d\mathbf{F})(u,w))-i^*(\bar{\nabla}d\mathbf{F})(\nabla_zu,w)-i^*(\bar{\nabla}d\mathbf{F})(u,\nabla_zw)\\
&=z(h(u,w))-h(\nabla_zu,w)-h(u,\nabla_zw)\\
&=(\nabla_zh)(u,w)\\
&=-2\hat{C}(z,u,w).
\end{align*}
So (\ref{Chat=Ahat}) follows. One similarly has (\ref{That=eta}) from (\ref{That2}).

\end{proof}

The above lemmas show that how to derive the centroaffine differential geometric structures of the indicatrices of Minkowski spaces from the Minkowski norms.
Conversely, a hyperovaloid $M$ with origin in its interior can define a Minkowski norm $\mathbf{F}$, such that $\mathbf{I_F}=M$.

\begin{lem}\label{hperovaloid to norm}
Let $V$ be an $n$ dimensional vector space. Let $M$ be a hyperovaloid in $V$ with origin in its interior.
Then there is a Minkowski norm $\mathbf{F}$ on $V$ such that $\mathbf{I_F}=M$.
\end{lem}

\begin{proof}
Since $M$ is strongly convex, then for each $\tilde{v}\in V_0$, the ray $\{t\tilde{v}|t>0\}$ has a unique intersection point $v$ with $M$.
Let $\tilde{v}=\lambda v$. then we can define a function $\mathbf{F}:V_0\rightarrow (0,+\infty)$ as
\begin{align*}
\mathbf{F}(\tilde{v})=\lambda.
\end{align*}
It is clear that $\mathbf{F}$ is positive homogeneous degree $1$.
Let $(U,\psi)$ be a coordinate chart of $M$ with the coordinate map $\psi:U\rightarrow \mathbb{R}^{n-1}$.
The scalar product $s:\mathbb{R}^{+}\times V\rightarrow V$ is clearly smooth.
Then the restriction of $s$ on $\mathbb{R}^{+}\times U$ is also smooth as $M$ is a smooth imbeded submanifold.
We will denote that $\mathbb{R}^{+}U:=s(\mathbb{R}^{+}\times U)$.
Since the Jacobian of $s$ at $(\lambda, v)$ is $J=\lambda^{n-1}(v,dv)=(-1)^n\lambda^{n-1}(y,dy)$,
where $y=-v$ denotes the centroaffine norm of $M$ at $v$. Then $J$ is nonsingular.
By inverse function theorem, $s:\mathbb{R}^{+}\times U\rightarrow\mathbb{R}^{+}U$ is a diffeomorphism.
Let $\tilde{\phi}=(i\times\psi)\circ s^{-1}$ be a map from $\mathbb{R}^{+}U$ to $\mathbb{R}\times\mathbb{R}^{n-1}=\mathbb{R}^{n}$.
Then $\tilde{\phi}$ gives a coordinate map of $\mathbb{R}^{+}U$. It is clear that
\begin{align*}
\mathbf{F}\circ\tilde{\phi}^{-1}(\lambda,\psi(U))=\lambda,
\end{align*}
then $\mathbf{F}$ is smooth on $V_0$.

Similar to Lemma \ref{lemma 2.3}, $\{-v,-d\mathbf{F}\}$ is the centroaffine normalization of $M$, where $v\in M$.
Moreover, we have
\begin{align*}
s^*\bar{g}=s^*\left(\bar{\nabla}d\left[\frac{1}{2}\mathbf{F}^2\right]\right)=d\lambda\otimes d\lambda+\lambda h,
\end{align*}
where $h$ is the induced Riemannian metric of $M$ with respect to the centroaffine normalization. So we have proved that $\mathbf{F}$ is a Minkowski norm on $V$.
$\mathbf{I_F}=M$ holds by definition.
\end{proof}

\subsection{Proof of Theorem \ref{theo equivalence}}
\

\begin{proof}
The sufficient part follows form Lemma \ref{lemma 2.2}. Now we are going to prove the necessary part.
It is clear that $f$ induces a diffeomorphism between the indicatrices $f:\mathbf{I_F}_1\rightarrow \mathbf{I_F}_2$.
As a consequence of  Lemma \ref{lemma 2.4}, \ref{lemma 2.5} and \ref{global unique c}, $f$ is just a linear homomorphism.
Then the proof is complete by applying Lemma \ref{lemma 2.1}.
\end{proof}

\section{Chern connection and Cartan-type one form}

\subsection{Chern connection in Finsler geometry}
\

Let $M$ be an $n$ dimensional smooth manifold and $\pi:TM\to M$
the tangent bundle of $M$. Let $(U;\phi(x)=(x^1,x^2,\ldots,x^{n}))$ be a
local coordinate system on an open subset $U$ of $M$. Then by the
standard procedure one gets a local coordinate system
$\psi(x,y)=(x^1,\ldots,x^{n},y^1,\ldots,y^{n})$ on $\pi^{-1}(U)$. Set
$TM_0=TM\setminus0$, where $0$ denotes the zero section of $TM$.
Then $\psi(x,y)$ with $y\neq 0$ is a local coordinate system on $TM_0$.

\begin{definition}\label{definition of Finsler manifolds}
A Finsler structure on $M$ is a continue function
$\mathbf{F}:TM\rightarrow\mathbb{R}$, which is smooth on $TM_0$, such that $\mathbf{F}_{T_xM}$ is a Minkowski norm for each $x\in M$. A manifold $M$ with a
Finsler structure $\mathbf{F}$ is called a Finsler manifold, and denoted by
$(M,\mathbf{F})$.

Let $F=\mathbf{F}\circ\psi^{-1}$, then $F$ is a smooth function of $2n$ variables. By the definition of Minkowski spaces, the $n\times n$ matrix
\begin{align*}
(g_{ij}):=\left(\frac{1}{2}[F^2]_{y^iy^j}\right)
\end{align*}
is positive definite everywhere.
\end{definition}

Using the Finsler structure $\mathbf{F}$ of a Finsler manifold $(M,\mathbf{F})$, one
can compute the energy variation of curves on $M$. The following
important data in Finsler geometry appears naturally in this
process:
\begin{align*}
G^{i}=\frac{1}{4}g^{ij}\left(\left[F^2\right]_{y^{j}x^{k}}y^{k}-\left[F^2\right]_{x^{j}}\right),
\end{align*}
where $(g^{ij})=(g_{ij})^{-1}$. It is clear that
\begin{align}\label{hom of G}
G^i(x,\lambda y)=\lambda^2G^i(x,y),\quad \lambda>0, ~i=1,\ldots, n.
\end{align}
The spray $\mathbf{G}$ or the Reeb field of the Finsler manifold $(M,\mathbf{F})$ is defined as a special smooth vector field on $TM_0$ as follows
\begin{align*}
\mathbf{G}=y^i\frac{\partial}{\partial x^i}-2G^i\frac{\partial}{\partial y^i}.
\end{align*}
In literatures, $G^i$'s are called the spray coefficients.

Set
\begin{align*}
\frac{\delta}{\delta x^i}:=\frac{\partial}{\partial
x^{i}}-\frac{\partial G^{j}}{\partial y^i}\frac{\partial}{\partial y^j},\quad \frac{\delta}{\delta y^i}:=F\frac{\partial}{\partial
y^{i}}.
\end{align*}
Clearly, the vectors
\begin{align}
\left\{\frac{\delta}{\delta x^1},\frac{\delta}{\delta x^2},\ldots,
\frac{\delta}{\delta x^{n}}, \frac{\delta}{\delta
y^1},\frac{\delta}{\delta y^2},\ldots, \frac{\delta}{\delta
y^{n}}\right\}\label{basis of TM_0}
\end{align}
form a local tangent frame of $TM_0$.
Then $T(TM_0)$ admits a splitting induced from the Finsler structure $\mathbf{F}$,
$$T(TM_0)=H(TM_0)\oplus V(TM_0),$$
where $H(TM_0)={\rm span}\left\{\frac{\delta}{\delta x^1},\ldots,
\frac{\delta}{\delta x^{n}}\right\}$, and $V(TM_0)={\rm span}\left\{\frac{\delta}{\delta y^1},\ldots,
\frac{\delta}{\delta y^{n}}\right\}$.
Then one gets a well-defined linear map $J:T(TM_0)\to T(TM_0)$
\begin{align*}
J\left(\frac{\delta}{\delta x^i}\right)=\frac{\delta}{\delta
y^i},\quad J\left(\frac{\delta}{\delta
y^i}\right)=-\frac{\delta}{\delta x^i},
\end{align*}
which is in fact an almost complex structure on $TM_0$. Let
\begin{align*}
\left\{\delta x^1,\delta x^2,\ldots,\delta x^{n},\delta y^1,\delta
y^2,\ldots,\delta y^{n}\right\}
\end{align*}
be the dual frame of (\ref{basis of TM_0}). One has
\begin{align*}
\delta x^i=dx^i,\quad \delta y^i={1\over
F}\left(dy^{i}+\frac{\partial G^{i}}{\partial
y^j}dx^{j}\right),
\end{align*}
and
\begin{align}
J^*(\delta x^i)=-\delta y^i,\quad J^*(\delta y^i)=\delta
x^i,\label{JiJ}
\end{align}
where $J^*$ denotes the dual map of $J$.

Now the fundamental tensor $g=g_{ij}dx^i\otimes dx^j$ defines an Euclidean metric on
the pull back bundle $\pi^*TM$ over $TM_0$. Note
that $\pi^*TM$ admits a distinguished global section
$l:TM_0\to\pi^{\ast}TM$, which is defined by
\begin{align*}
l(x,y)=\left(x,y,\frac{y}{F(x,y)}\right).
\end{align*}

For any local orthonormal frame field $\left\{e_1,\ldots,e_{n}\right\}$ of
$(\pi^*TM,g)$ with $e_{n}=l$, let $\{\omega^{1},\cdots,\omega^{n}\}$
be the dual frame. Clearly, $\omega^i$'s can be viewed
naturally as (local) one forms on $TM_0$. Here $\omega^{n}$,
the so called Hilbert form, is a globally defined one form and $\omega^{n}=F_{y^i}\delta x^i$. Set
\begin{align*}
\omega^{n+i}=J^*(\omega^i),\quad i=1,2,\ldots, n,
\end{align*}
and
\begin{align}
\theta=\left\{\omega^1,\ldots,\omega^n,\omega^{n+1},\ldots,\omega^{2n}\right\},
                                                         \label{total cobasis of SM}
\end{align}
where $\omega^{2n}=-F_{y^i}\delta y^i=-d\log F$.
$\theta$ forms a local coframe of $TM_0$. The tensor
\begin{align*}
g^{T(TM_0)}=\sum_{i=1}^{n}\omega^i\otimes\omega^i+\sum_{i=1}^{n}\omega^{n+i}\otimes\omega^{n+i}
\end{align*}
gives raise a Riemannian metric on $TM_0$.
Let
$\{\mathbf{e}_1,\ldots,\mathbf{e}_n,\mathbf{e}_{n+1},\ldots,\mathbf{e}_{2n}\}$
denote the dual frame of $\theta$. It is clear that
\begin{align}\label{mathbf e}
H(TM_0)={\rm span}\{\mathbf{e}_1,\ldots,\mathbf{e}_n\}.
\end{align}

Write that
\begin{align*}
\omega^j=v^{j}_i\delta x^i,\quad{\rm and\quad so}\quad\omega^{n+j}=J^*(v^{j}_i\delta
x^i)=-v^{j}_i\delta y^i.
\end{align*}
Then one has
\begin{align}
\mathbf{e}_i=u_{i}^j\frac{\delta}{\delta x^j}\quad{\rm and}\quad\mathbf{e}_{n+i}
=-u_{i}^j\frac{\delta}{\delta y^j},\label{ei to delta delta x}
\end{align}
where $(u_{j}^i)=(v^{j}_i)^{-1}$. One also notes that $v^{n}_i=F_{y^i}$
and $u^i_n=\frac{y^i}{F}$.

Let $\nabla^{\rm Ch}$ be the Chern connection, i.e.,
\begin{align*}
\nabla^{\rm Ch}:\Omega^*(TM_0;H(TM_0))\rightarrow \Omega^{*+1}(TM_0;H(TM_0)).
\end{align*}
It is proved in \cite{FL} that the symmetrization of Chern connection is just the Cartan connection $\nabla^{\rm Ca}$.
The difference between $\nabla^{\rm Ca}$ and $\nabla^{\rm Ch}$ will be referred as the Cartan endomorphism,
\begin{align*}
H=\nabla^{\rm Ca}-\nabla^{\rm Ch}\quad\in\Omega^1(TM_0,{\rm
End}(H(TM_0))).
\end{align*}
Set $H=H_{ij}\omega^j\otimes\mathbf{e}_i$. By Lemma 3 and Lemma 4 in \cite{FL}, $H_{ij}=H_{ji}=H_{ij\gamma}\omega^{n+\gamma}$ is locally determined with respect to (\ref{ei to delta delta x}) by
\begin{align*}
H_{ij\gamma}=-A_{pqk}u_i^pu_j^qu_{\gamma}^k,
\end{align*}
where $A_{ijk}=\frac{1}{4}F[F^2]_{y^iy^jy^k}$.

Let $\bm{\omega}=(\omega_j^i)$ be the connection matrix of the Chern
connection with respect to the local orthonormal frame field (\ref{mathbf e}),
i.e.,
\begin{align*}
\D{\rm Ch}\mathbf{e}_i=\omega_i^j\mathbf{e}_j.
\end{align*}
\begin{lem}[\cite{BaoChernShen,ChernShen,Mo}]\label{sturcture eq}
The connection matrix $\bm{\omega}=(\omega_j^i)$ of $\nabla^{\rm Ch}$ is determined by the following structure equations,
\begin{equation}\left\{
\begin{aligned}
&d\vartheta=-\bm{\omega}\wedge\vartheta,\\
&\bm{\omega}+\bm{\omega}^t=-2H,
\end{aligned}\right.\label{Chern connection structure eq. matrix}
\end{equation}
where $\vartheta=(\omega^1,\ldots,\omega^{n})^t$. Furthermore,
$$\omega_{\alpha}^{n}=-\omega^{\alpha}_{n}=\omega^{n+\alpha},\quad{\rm and}\quad \omega^{n}_{n}=0.$$
\end{lem}

\begin{remark}
The first equation in (\ref{Chern connection structure eq. matrix}) is described as torsion freeness of Chern connection.
And the second equation in (\ref{Chern connection structure eq. matrix}) implies that Chern connection is almost preserving metric.
This connection is constructed originally by Chern in the study of local equivalence problem in Finsler spaces \cite{Chern}.
In \cite{FL}, we proved that Chern connection is just the Bott connection on $H(TM_0)$ by the theory of foliation (\cite{Zhang}).
\end{remark}

Let $R^{\rm Ch}=\left(\nabla^{\rm Ch}\right)^2$  be the curvature of $\nabla^{\rm Ch}$.
Let $\Omega=\left(\Omega_j^i\right)$ be the curvature forms of $R^{\rm Ch}$. Then
\begin{align*}
\Omega_j^i=d\omega_j^i-\omega_j^k\wedge\omega_k^i.
\end{align*}
From Lemma \ref{sturcture eq},
\begin{align*}
\Omega_j^i=\frac{1}{2}R_{j~kl}^{~i}\omega^k\wedge\omega^l+P_{j~k\gamma}^{~i}\omega^k\wedge\omega^{n+\gamma},
\end{align*}
where $R_{j~kl}^{~i}=-R_{j~lk}^{~i}.$
The the ``h-h'' part
\begin{align*}
R:=\frac{1}{2}R_{j~kl}^{~i}(\omega^k\wedge\omega^l)\otimes\omega^j\otimes\mathbf{e}_i
\end{align*}
of $R^{\rm Ch}$ will be referred as the Chern-Riemann curvature.
The ``h-v'' part
\begin{align*}
P:=P_{j~k\gamma}^{~i}(\omega^k\wedge\omega^{n+\gamma})\otimes\omega^j\otimes\mathbf{e}_i
\end{align*}
will be called the Chern-Minkowski curvature.
The Landsberg curvature is defined as
$$L:=P_{n~k\gamma}^{~i}(\omega^k\wedge\omega^{n+\gamma})\otimes\mathbf{e}_i,$$
the mean Landsberg curvature is defined by $\mathbf{J}={\rm tr}L$.
If a Finsler manifold satisfies $P=0$, $L=0$ or $\mathbf{J}=0$, then it is called a Berwald, Landsberg or weak Landsberg manifold, respectively.

\begin{lem}[\cite{BaoChernShen,ChernShen,Mo}]\label{bianchi 1}
One has the following Bianchi identities induced from the torsion free property of Chern connection.
\begin{equation*}
R_{j~kl}^{~i}+R_{k~lj}^{~i}+R_{l~jk}^{~i}=0,
\end{equation*}
\begin{equation*}
P_{j~k\gamma}^{~i}=P_{k~j\gamma}^{~i}. \label{P 13 symmetry}
\end{equation*}
\end{lem}

Some more involved Bianchi identities derived from the almost metric preserving properties of Chern connection.

\begin{lem}[\cite{BaoChernShen,ChernShen,Mo}]\label{bianchi 2}
$$R_{j~kl}^{~i}+R_{i~kl}^{~j}-2H_{ij\gamma}R_{n~kl}^{~\gamma}=0,$$
$$H_{ij\gamma;\alpha}=H_{ij\alpha;\gamma},$$
and
\begin{align}\label{b5}
P_{j~k\gamma}^{~i}+P_{i~k\gamma}^{~j}-2H_{ij\beta}P_{n~k\gamma}^{~\beta}+2H_{ij\gamma|k}=0.
\end{align}
where the following notation is adopted
\begin{align*}
&dH_{ij\gamma}-H_{ik\gamma}\omega_j^k-H_{jk\gamma}\omega_i^k-H_{ij\beta}\omega_{\gamma}^{\beta}=:H_{ij\gamma|k}\omega^k+H_{ij\gamma;\alpha}\omega^{n+\alpha}.
\end{align*}
\end{lem}

As consequences Lemma \ref{bianchi 1} and Lemma \ref{bianchi 2}, we have the following formulas about the Chern-Minkowski curvature.
\begin{lem}[\cite{BaoChernShen,ChernShen,Mo}]\label{bianchi 3}
\begin{align}\label{b6}
P_{n~k\gamma}^{~i}=-H_{ki\gamma|n},
\end{align}
and
\begin{equation}\label{b7}
P_{j~k\gamma}^{~i}=H_{ij\beta}H_{k\beta\gamma|n}+H_{ki\beta}H_{j\beta\gamma|n}-H_{jk\beta}H_{i\beta\gamma|n}-H_{ij\gamma|k}-H_{ki\gamma|j}+H_{jk\gamma|i}.
\end{equation}
\end{lem}

We define two symmetric tensor fields on $TM_0$ as follows
\begin{align*}
\hat{g}=\sum_{j=1}^{n}\omega^{n+j}\otimes\omega^{n+j},
\quad {\rm and}\quad
\hat{A}=H_{\alpha\beta\gamma}\omega^{n+\alpha}\otimes\omega^{n+\beta}\otimes\omega^{n+\gamma}.
\end{align*}
For any $p\in M$, the restrictions of $\hat{g}$ and $\hat{A}$ on each fiber $T_pM\setminus\{0\}$ gives the fundamental form and Cartan tensor of the Minkowski space $(T_pM, \mathbf{F}_{T_pM})$, respectively.

\begin{prop}\label{horizontal lie g A}
The Lie derivatives of $\hat{g}$ and $\hat{A}$  along horizontal directions  are given by,
\begin{align}\label{horizontal lie d of hat g}
\mathcal{L}_{\mathbf{e}_i}\hat{g}\equiv-2\sum_{\alpha,\beta=1}^{n-1}P_{n~i\alpha}^{~\beta}\omega^{n+\alpha}\otimes\omega^{n+\beta}~({\rm mod}~\omega^1,\ldots,\omega^n),
\end{align}
and
\begin{align}\label{horizontal lie d of hat A}
\mathcal{L}_{\mathbf{e}_i}\hat{A}
\equiv Z_{\alpha\beta\gamma i}\omega^{n+\alpha}\otimes\omega^{n+\beta}\otimes\omega^{n+\gamma}~({\rm mod}~\omega^1,\ldots,\omega^n),
\end{align}
where $Z_{\alpha\beta\gamma i}=H_{\alpha\beta\gamma|i}-H_{\mu\beta\gamma}P_{n~i\alpha}^{~\mu}-H_{\alpha\mu\gamma}P_{n~i\beta}^{~\mu}-H_{\alpha\beta\mu}P_{n~i\gamma}^{~\mu}$.

Let $Z$ be the tensor defined by $Z_{\alpha\beta\gamma i}$, then $Z=0$ if and only if $P=0$.
\end{prop}
\begin{proof}
Since that
\begin{align*}
\Omega_n^{\alpha}=d\omega_n^{\alpha}-\omega_n^{\beta}\wedge\omega_{\beta}^{\alpha}=-d\omega^{n+\alpha}+\omega^{n+\beta}\wedge\omega^{\alpha}_{\beta},
\end{align*}
then
\begin{align*}
d\omega^{n+\alpha}&=-\Omega_n^{\alpha}+\omega^{n+\beta}\wedge\omega^{\alpha}_{\beta}\\
&=-\frac{1}{2}R_{n~jk}^{~\alpha}\omega^j\wedge\omega^k-P_{n~j\beta}^{~\alpha}\omega^j\wedge\omega^{n+\beta}+\omega^{n+\beta}\wedge\omega^{\alpha}_{\beta}.
\end{align*}
By Cartan homotopy formula, we have
\begin{align*}
\mathcal{L}_{\mathbf{e}_i}\omega^{n+\alpha}&=\left(i_{\mathbf{e}_i}d+di_{\mathbf{e}_i}\right)\omega^{n+\alpha}\\
&=i_{\mathbf{e}_i}\left[-\frac{1}{2}R_{n~jk}^{~\alpha}\omega^j\wedge\omega^k-P_{n~j\beta}^{~\alpha}\omega^j\wedge\omega^{n+\beta}+\omega^{n+\beta}\wedge\omega^{\alpha}_{\beta}\right]\\
&=-R_{n~ik}^{~\alpha}\omega^k-P_{n~i\beta}^{~\alpha}\omega^{n+\beta}-\omega^{\alpha}_{\beta}(\mathbf{e}_i)\omega^{n+\beta}.
\end{align*}
Using the fact $\mathbf{e}_i(F)=0$ and $\omega^{2n}=-d\log F$, one gets $\mathcal{L}_{\mathbf{e}_i}\omega^{2n}=0$.
So we have
\begin{align*}
&\mathcal{L}_{\mathbf{e}_i}(\sum_{j=1}^{n}\omega^{n+j}\otimes\omega^{n+j})
=\mathcal{L}_{\mathbf{e}_i}(\sum_{\alpha=1}^{n-1}\omega^{n+\alpha}\otimes\omega^{n+\alpha})\\
=&\sum_{\alpha=1}^{n-1}\left[\left(\mathcal{L}_{\mathbf{e}_i}\omega^{n+\alpha}\right)\otimes\omega^{n+\alpha}+\omega^{n+\alpha}\otimes\left(\mathcal{L}_{\mathbf{e}_i}\omega^{n+\alpha}\right)\right]\\
=&\sum_{\alpha=1}^{n-1}\left[\left(-R_{n~ik}^{~\alpha}\omega^k-P_{n~i\beta}^{~\alpha}\omega^{n+\beta}-\omega^{\alpha}_{\beta}(\mathbf{e}_i)\omega^{n+\beta}\right)\otimes\omega^{n+\alpha}\right.\\
&\qquad\left.+\omega^{n+\alpha}\otimes\left(-R_{n~ik}^{~\alpha}\omega^k-P_{n~i\beta}^{~\alpha}\omega^{n+\beta}-\omega^{\alpha}_{\beta}(\mathbf{e}_i)\omega^{n+\beta}\right)\right]\\
\equiv&-\sum_{\alpha=1}^{n-1}\left(P_{n~i\beta}^{~\alpha}+\omega^{\alpha}_{\beta}(\mathbf{e}_i)\right)\left(\omega^{n+\beta}\otimes\omega^{n+\alpha}+\omega^{n+\alpha}\otimes\omega^{n+\beta}\right)\\
\equiv&-2\sum_{\alpha,\beta=1}^{n-1}P_{n~i\alpha}^{~\beta}\omega^{n+\alpha}\otimes\omega^{n+\beta}~({\rm mod}~\omega^1,\ldots,\omega^n),
\end{align*}
where we have use the facts derived from Lemma \ref{sturcture eq} and \ref{bianchi 3},
$$P_{n~i\beta}^{~\alpha}=P_{n~i\alpha}^{~\beta},\quad{\rm and}\quad\omega^{\beta}_{\alpha}+\omega^{\alpha}_{\beta}=H_{\alpha\beta\gamma}\omega^{n+\gamma}.$$

Furthermore, we have
\begin{align*}
&\mathcal{L}_{\mathbf{e}_i}\left(H_{\alpha\beta\gamma}\omega^{n+\alpha}\otimes\omega^{n+\beta}\otimes\omega^{n+\gamma}\right)\\
=&\mathbf{e}_i(H_{\alpha\beta\gamma})\omega^{n+\alpha}\otimes\omega^{n+\beta}\otimes\omega^{n+\gamma}+H_{\alpha\beta\gamma}\left(\mathcal{L}_{\mathbf{e}_i}\omega^{n+\alpha}\right)\otimes\omega^{n+\beta}\otimes\omega^{n+\gamma}\\
&\quad+H_{\alpha\beta\gamma}\omega^{n+\alpha}\otimes\left(\mathcal{L}_{\mathbf{e}_i}\omega^{n+\beta}\right)\otimes\omega^{n+\gamma}+H_{\alpha\beta\gamma}\omega^{n+\alpha}\otimes\omega^{n+\beta}\otimes\left(\mathcal{L}_{\mathbf{e}_i}\omega^{n+\gamma}\right)\\
\equiv&\mathbf{e}_i(H_{\alpha\beta\gamma})\omega^{n+\alpha}\otimes\omega^{n+\beta}\otimes\omega^{n+\gamma}-H_{\alpha\beta\gamma}\left(P_{n~i\mu}^{~\alpha}+\omega^{\alpha}_{\mu}(\mathbf{e}_i)\right)\omega^{n+\mu}\otimes\omega^{n+\beta}\otimes\omega^{n+\gamma}\\
&\quad-H_{\alpha\beta\gamma}\omega^{n+\alpha}\otimes\left(P_{n~i\mu}^{~\beta}+\omega^{\beta}_{\mu}(\mathbf{e}_i)\right)\omega^{n+\mu}\otimes\omega^{n+\gamma}\\
&\quad-H_{\alpha\beta\gamma}\omega^{n+\alpha}\otimes\omega^{n+\beta}\otimes\left(P_{n~i\mu}^{~\gamma}+\omega^{\gamma}_{\mu}(\mathbf{e}_i)\right)\omega^{n+\mu}\\
\equiv&\mathbf{e}_i(H_{\alpha\beta\gamma})\omega^{n+\alpha}\otimes\omega^{n+\beta}\otimes\omega^{n+\gamma}-H_{\alpha\beta\gamma}\omega^{\alpha}_{\mu}(\mathbf{e}_i)\omega^{n+\mu}\otimes\omega^{n+\beta}\otimes\omega^{n+\gamma}\\
&\quad-H_{\alpha\beta\gamma}\omega^{\beta}_{\mu}(\mathbf{e}_i)\omega^{n+\alpha}\otimes\omega^{n+\mu}\otimes\omega^{n+\gamma}-H_{\alpha\beta\gamma}\omega^{\gamma}_{\mu}(\mathbf{e}_i)\omega^{n+\alpha}\otimes\omega^{n+\beta}\otimes\omega^{n+\mu}\\
&\quad-H_{\alpha\beta\gamma}P_{n~i\mu}^{~\alpha}\omega^{n+\mu}\otimes\omega^{n+\beta}\otimes\omega^{n+\gamma}-H_{\alpha\beta\gamma}P_{n~i\mu}^{~\beta}\omega^{n+\alpha}\otimes\omega^{n+\mu}\otimes\omega^{n+\gamma}\\
&\quad-H_{\alpha\beta\gamma}P_{n~i\mu}^{~\gamma}\omega^{n+\alpha}\otimes\omega^{n+\beta}\otimes\omega^{n+\mu}\\
\equiv&\left(\mathbf{e}_i(H_{\alpha\beta\gamma})-H_{\mu\beta\gamma}\omega^{\mu}_{\alpha}(\mathbf{e}_i)-H_{\alpha\mu\gamma}\omega^{\mu}_{\beta}(\mathbf{e}_i)-H_{\alpha\beta\mu}\omega^{\mu}_{\gamma}(\mathbf{e}_i)\right)\omega^{n+\alpha}\otimes\omega^{n+\beta}\otimes\omega^{n+\gamma}\\
&\quad-\left(H_{\mu\beta\gamma}P_{n~i\alpha}^{~\mu}+H_{\alpha\mu\gamma}P_{n~i\beta}^{~\mu}+H_{\alpha\beta\mu}P_{n~i\gamma}^{~\mu}\right)\omega^{n+\alpha}\otimes\omega^{n+\beta}\otimes\omega^{n+\gamma}\\
\equiv&\left(H_{\alpha\beta\gamma|i}-H_{\mu\beta\gamma}P_{n~i\alpha}^{~\mu}-H_{\alpha\mu\gamma}P_{n~i\beta}^{~\mu}-H_{\alpha\beta\mu}P_{n~i\gamma}^{~\mu}\right)\omega^{n+\alpha}\otimes\omega^{n+\beta}\otimes\omega^{n+\gamma}~({\rm mod}~\omega^1,\ldots,\omega^n).
\end{align*}

From the definition of $Z$ and Lemma \ref{bianchi 3}, $Z=0$ implies $Z_{\alpha\beta\gamma|n}=-P^{\beta}_{n~\alpha\gamma}=0$ and $H_{\alpha\beta\gamma|i}=0$.
Using Lemma \ref{bianchi 3} once more, $Z=0$ implies $P=0$. The inverse statement is clear.
\end{proof}

\subsection{The Cartan-type form and its properties}
\

Let
\begin{align*}
\eta={\rm tr}[H]~\in\Omega^1(TM_0).
\end{align*}
It is referred as the Cartan-type in \cite{FL}. Then Cartan-type form has a local expression
\begin{align*}
\eta=\sum_{i=1}^nH_{ii\gamma}\omega^{n+\gamma}=:H_{\gamma}\omega^{n+\gamma}.
\end{align*}
It is clear that the restriction of $\eta$ on the fibers of $TM_0$ are just the Cartan forms for the corresponding Minkowski spaces.

\begin{remark}
In literature, Cartan from of a Finsler manifold is locally defined by $I=H_{\gamma}\omega^{\gamma}$. This is the reason why we call $\eta$ the Cartan-type form.
Some simple calculus show that $\eta$ and $I$ behavior differently. For example,  $dI=0$ means $M$ is a Riemannian manifold.
However, we will see that there are non-Riemannian metrics such that $d\eta=0$.
\end{remark}

Let $R^{\rm Ch}$ and $R^{\rm Ca}$ be the curvature of $\nabla^{\rm Ch}$ and $\nabla^{\rm Ca}$, respectively. Then
we have

\begin{prop} \label{lem deta}
The exterior differentiation of $\eta$ is given by
\begin{equation}
d\eta=-\tr [R^{\rm Ch}]. \label{deta}
\end{equation}
Then $d\eta$ has the local formula
\begin{equation*}
d\eta=d(H_{\gamma}\omega^{n+\gamma})=-\frac{1}{2}R_{i~kl}^{~i}\omega^k\wedge\omega^l-P_{i~k\gamma}^{~i}\omega^k\wedge\omega^{n+\gamma}.
\end{equation*}
\end{prop}
\begin{proof}
It is well known that the Cartan connection is metric-compatible.
Then $\tr [R^{\rm Ca}]=0$. So one has
\begin{align*}
0&=\tr [R^{\rm Ca}]=\tr [(\nabla^{\rm Ca})^2]\\
&=\tr \left[\left(\nabla^{\rm Ch}+H\right)^2\right]\\
&=\tr\left[\left(\nabla^{\rm Ch}\right)^2+[\nabla^{\rm
Ch},H]+[H,H]\right]\\
&=\tr\left[R^{\rm Ch}\right]+\tr\left[[\nabla^{\rm
Ch},H]\right]+\tr\left[[H,H]\right]
\end{align*}
where $[\cdot,\cdot]$ denotes the super bracket on $\Omega^*(TM_0,{\rm
End}(H(TM_0)))$. By the element facts in \cite{Zhang}, we have
$$\tr\left[[H,H]\right]=0,\quad {\rm and}\quad d\tr[H]=\tr\left[[\nabla^{\rm
Ch},H]\right].$$ So the proof is complete.
\end{proof}
From Proposition \ref{lem deta}, Berwald manifolds must have closed Cartan-type form.

\begin{prop}\label{horizontal lie eta}
The Lie derivatives of $\eta$ along horizontal vectors is given as follows,
\begin{align}\label{horizontal lie d of hat eta}
\mathcal{L}_{\mathbf{e}_i}\eta\equiv -i_{\mathbf{e}_i}{\rm tr}P~({\rm mod}~\omega^1,\ldots,\omega^n).
\end{align}
\end{prop}
\begin{proof}
Using Proposition \ref{lem deta}, we get
\begin{equation*}
  \mathcal{L}_{\mathbf{e}_i}\eta=(di_{\mathbf{e}_i}+i_{\mathbf{e}_i}d)\eta=i_{\mathbf{e}_i}d\eta\equiv -i_{\mathbf{e}_i}{\rm tr}P~({\rm mod}~\omega^1,\ldots,\omega^n).
\end{equation*}
\end{proof}

The following discussion gives another explanation of $\eta$ and its relation with the $S$-curvature.

On a local coordinate chart $(U;x^i)$, let $dV_M=\sigma(x)dx^1\wedge\cdots\wedge dx^n$
be any volume form on $M$.  The following important function
on $TM_0$ is well defined,
$$\tau=\ln\frac{\sqrt{\det{g_{ij}}}}{\sigma(x)}.$$
$\tau$ is called the distortion of $(M,\mathbf{F})$. $\tau$ is a very
important invariant of the Finsler manfold, which is first
introduced by Zhongmin Shen. One refers to \cite{ChernShen}
for more discussion about the distortion $\tau$.

Let $\mathrm{l}=\mathbb{R}\mathbf{G}$ be the line bundle generated
by the Reeb vector field $\mathbf{G}=\mathbf{F}\mathbf{e}_n$. The quotient subbundle
$H(TM_0)/\mathrm{l}$ will be canonically chosen as the orthogonal
complement of $\mathrm{l}$ in $H(TM_0)$ with respect to $g$, i.e.,
$$\mathrm{l}^{\bot}=H(TM_0)/\mathrm{l}.$$

\begin{definition}
Using the Levi-Civita connection $\D{T(TM_0)}$ of the metric $g^{T(TM_0)}$, we introduce the following three
operator on $\Omega^*(TM_0)$.
\begin{equation*}
d^{\mathrm{l}^{\bot}}=\omega^{\alpha}\wedge\D{T^*(TM_0)}_{\mathbf{e}_{\alpha}},\quad
d^{\mathrm{l}}=\omega^{n}\wedge\D{T^*(TM_0)}_{\mathbf{e}_n},\quad
d^{V}=\omega^{n+i}\wedge\D{T^*(TM_0)}_{\mathbf{e}_{n+i}},
\end{equation*}
where $\D{T^*(TM_0)}$ denotes the Levi-Civita connection on the cotangent bundle $T^*(TM_0)$.
\end{definition}
It is obvious that the above operators are well defined. So the exterior differential operator $d$ on $TM_0$ splits to three parts
$$d=d^{\mathrm{l}^{\bot}}+d^{\mathrm{l}}+d^{V}.$$

\begin{lem}
$$d^{V}\tau=\eta.$$
\end{lem}

\begin{proof}
First we have
\begin{align*}
d^{V}\tau&\equiv d\tau\equiv d\ln\sqrt{\det{g_{ij}}}-d\ln\sigma(x)\equiv d\ln\sqrt{\det{g_{ij}}}~({\rm mod}~\omega^1,\ldots,\omega^n).
\end{align*}
Assume that $\{\omega^1,\ldots,\omega^n\}$ has the same
orientation with $M$. Let
$\omega=\omega^1\wedge\cdots\wedge\omega^n$, by Lemma \ref{sturcture eq}, one has
\begin{align*}
d\omega&=\sum_{i}(-1)^{i-1}d\omega^i\wedge\omega^1\wedge\cdots\wedge\widehat{\omega^i}\wedge\cdots\wedge\omega^n\\
&=\sum_{i}(-1)^{i-1}\omega^j\wedge\omega_j^i\wedge\omega^1\wedge\cdots\wedge\widehat{\omega^i}\wedge\cdots\wedge\omega^n\\
&=-(\sum_{i}\omega_i^i)\wedge\omega=\eta\wedge\omega.
\end{align*}

Since
$$\omega=\omega^1\wedge\cdots\wedge\omega^n=\sqrt{\det(g_{ij})}dx^1\wedge\cdots\wedge dx^n,$$
it follows that
\begin{align*}
d\omega=d\sqrt{\det(g_{ij})}\wedge dx^1\wedge\cdots\wedge dx^n
=d\ln\sqrt{\det{g_{ij}}}\wedge \omega.
\end{align*}
So one has
$$d\ln\sqrt{\det{g_{ij}}}\wedge \omega=\eta\wedge\omega.$$
Hence
$$\eta\equiv d\ln\sqrt{\det{g_{ij}}}\equiv d^V\tau~({\rm mod}~\omega^1,\ldots,\omega^n).$$

\end{proof}
Set
$$S=\mathbf{G}(\tau).$$
$S$ is called the $S$-curvature of the Finsler manifold $(M,\mathbf{F})$,
which is also introduced by Zhongmin Shen. For the detail of
$S$-curvature, one refers to
\cite{ChernShen,Shenbook1,Shenbook2,Shen}.

Set $\widetilde{S}=\frac{S}{\mathbf{F}}=\mathbf{e}_n(\tau).$
If $d^{V}\widetilde{S}=0$, then $(M,\mathbf{F})$ is called of isotropic
$S$-curvature. The following corollary is simple.

\begin{mcor}
\begin{equation}
d\tau=d^{\mathrm{l}^{\bot}}\tau+\tilde{S}\omega^n+\eta.\label{dtau}
\end{equation}

\end{mcor}

Then one has a
description of $(M,\mathbf{F})$ of isotropic $S$-curvature.


\begin{prop}
A Finsler manifold $(M,\mathbf{F})$ is of isotropic $S$-curvature if and only
if
\begin{equation*}
J^*(d^{\mathrm{l}^{\bot}}\tau)=\mathbf{J},
\end{equation*}
where $J^*$ is the mapping given by (\ref{JiJ}).
\end{prop}

\begin{proof}

Using Proposition \ref{lem deta}, the exterior differentiation of (\ref{dtau}) gives
\begin{align*}
0=&d^2\tau=d(d^{\mathrm{l}^{\bot}}\tau)+d(\tilde{S}\omega^n)+d\eta\\
&=d(e_{\alpha}(\tau)\omega^{\alpha})+d\tilde{S}\wedge\omega^n+\tilde{S}d\omega^n+d\eta\\
&=d(e_{\alpha}(\tau))\wedge\omega^{\alpha}+e_{\alpha}(\tau)d\omega^{\alpha}+d^{\mathrm{l}^{\bot}}\tilde{S}\wedge\omega^n+
d^{\mathrm{l}}\tilde{S}\wedge\omega^n+d^V\tilde{S}\wedge\omega^n+\tilde{S}d\omega^n+d\eta\\
&=e_{\alpha}(e_{\beta}(\tau))\omega^{\alpha}\wedge\omega^{\beta}+e_{n}(e_{\beta}(\tau))\omega^{n}\wedge\omega^{\beta}+e_{n+\alpha}(e_{\beta}(\tau))\omega^{n+\alpha}\wedge\omega^{\beta}\\
&+e_{\alpha}(\tau)\omega^{j}\wedge\omega^{\alpha}_{j}+d^{\mathrm{l}^{\bot}}\tilde{S}\wedge\omega^n+
d^V\tilde{S}\wedge\omega^n+\tilde{S}\omega^{\alpha}\wedge\omega^{n+\alpha}\\
&-\frac{1}{2}R^{~i}_{i~jk}\omega^j\wedge\omega^k-P^{~i}_{i~j\gamma}\omega^j\wedge\omega^{n+\gamma}\\
&=e_{\alpha}(e_{\beta}(\tau))\omega^{\alpha}\wedge\omega^{\beta}+e_{n}(e_{\beta}(\tau))\omega^{n}\wedge\omega^{\beta}+e_{n+\alpha}(e_{\beta}(\tau))\omega^{n+\alpha}\wedge\omega^{\beta}\\
&+e_{\alpha}(\tau)\omega^{n}\wedge\omega^{n+\alpha}+e_{\alpha}(\tau)\omega^{\beta}\wedge\omega^{\alpha}_{\beta}+d^{\mathrm{l}^{\bot}}\tilde{S}\wedge\omega^n+
d^V\tilde{S}\wedge\omega^n+\tilde{S}\omega^{\alpha}\wedge\omega^{n+\alpha}\\
&-\frac{1}{2}R^{~i}_{i~jk}\omega^j\wedge\omega^k-P^{~i}_{i~j\gamma}\omega^j\wedge\omega^{n+\gamma}.
\end{align*}

It follows that
\begin{align*}
d^V\tilde{S}\wedge\omega^n-e_{\alpha}(\tau)\omega^{n+\alpha}\wedge\omega^{n}-P^{~i}_{i~n\gamma}\omega^n\wedge\omega^{n+\gamma}=0.
\end{align*}
So
\begin{align*}
d^V\tilde{S}&=e_{\alpha}(\tau)\omega^{n+\alpha}-P^{~i}_{i~n\alpha}\omega^{n+\alpha}\\
&=J^*(d^{\mathrm{l}^{\bot}}\tau)-\mathbf{J}.
\end{align*}
\end{proof}

\begin{prop}\label{prop eta=dtau}
For a Finsler manifold $(M,\mathbf{F})$, let $\tau$ be the distortion with respect to a given volume element of $M$.
Then $\eta=d\tau$ if and only if $S=0$ and $\mathbf{J}=0$.
\end{prop}
\begin{proof}
If $\eta=d\tau$, then $d^{\mathrm{l}^{\bot}}\tau=0$ and $S=0$ directly holds. Since $S=0$ obviously implies $S$ is isotropic, then $\mathbf{J}=0$.
The proof of the converse statement is similarly.
\end{proof}

\section{Characterization of some special Finsler manifolds}

In this section, we will use parallel transport to characterize Finsler manifolds with some special curvature properties.
For details of the parallel transport in Finsler geometry, one can refer to \cite{Aikou10,Bao,ChernShen}.

Let $(M,\mathbf{F})$ be a Finsler manifold. Let $\sigma:[0,1]\rightarrow M$ be a piecewise smooth curve from $\sigma(0)=p$ to $\sigma(1)=q$,
then we have the parallel transport
\begin{equation*}
  P_{\sigma,t}:T_pM\rightarrow T_{\sigma(t)}M, \quad {\rm for}~\forall~t\in[0,1].
\end{equation*}
Furthermore, the restricted map $P_{\sigma,t}:T_pM\setminus\{0\}\rightarrow
T_{\sigma(t)}M\setminus\{0\}$ is a norm preserving diffeomorphism,
and satisfies
\begin{align}\label{radial linear of P}
P_{\sigma,t}(\lambda y)=\lambda P_{\sigma,t}(y), \quad \forall ~\lambda>0, \forall~y\in T_pM\setminus\{0\}.
\end{align}

Let $T$ be a vertical covariant tensor field on $TM_0$, i.e. $T\equiv 0~({\rm mod}~\omega^{n+1},\ldots,\omega^{2n})$.
$T$ is called to be preserved by parallel transport along the curve $\sigma$ if
\begin{equation*}
  P_{\sigma,t}^*T_{\sigma(t)}=T_{p}\quad {\rm for}~\forall~t\in[0,1],
\end{equation*}
where $T_x=i_x^*T$ denotes the restriction of $T$ on $T_xM\setminus\{0\}$, $i_x:T_xM\hookrightarrow TM$ is the embedding mapping for any $x\in M$.
\begin{lem}\label{T}
Let $T$ be a vertical covariant tensor field on $TM_0$.
Then $T$ is preserved by the parallel transports  if and only if
\begin{equation*}
  \mathcal{L}_{X}T\equiv 0~({\rm mod}~\omega^1,\ldots,\omega^n),
\end{equation*}
for any horizontal vector field $X\in\Gamma(H(TM_0))$.
\end{lem}
\begin{proof}
If $X$ is horizontal and $T$ is vertical, then $\mathcal{L}_{X}T$ is $C^{\infty}(TM_0)$-linear in $X$.
So we only need to deal with such $X$ obtained from horizontal lift.

  The problem is local, so we assume $\sigma:(-2\epsilon,2\epsilon)\rightarrow M$ is a smooth curve passing through $\sigma(0)=p\in M$, where $\epsilon>0$.
  On $\pi^{-1}(\sigma):=\cup_tT_{\sigma(t)}M$, let $X$ be the horizontal lift vector filed of the tangent vector field $\dot{\sigma}$ of $\sigma$.
  The horizontal lift of $\sigma$ is just the integral curves of $X$.
  Under a local coordinate system, the integral curves $\hat{\sigma}(t)=(\sigma(t),y(t))=(\sigma^i(t);y^i(t))$ of $X$ satisfies the following first order ODE's,
  \begin{align} \label{equation of parellel vector field}
\frac{dy^i}{dt}+\frac{d\sigma^j}{dt}\frac{\partial G^i}{\partial y^j}(\sigma(t),y(t))=0,\quad i=1,\ldots,n.
\end{align}
By (\ref{hom of G}), (\ref{equation of parellel vector field}) and the theory of ODE's, $X$ defines a local 1-parametric group
\begin{align*}
  \varphi:(-\epsilon,\epsilon)\times \pi^{-1}(\tilde{\sigma})\rightarrow \pi^{-1}(\sigma),
\end{align*}
where $\tilde{\sigma}$ is the restriction of $\sigma$ on $(-\epsilon,\epsilon)$. Set $\varphi_t=\varphi(t,\cdot)$,
by the definition of parallel transport, we have
\begin{equation*}
 \varphi_t(y)=P_{\sigma,t}(y), \quad {\rm for} ~\forall~y\in T_pM.
\end{equation*}
In other words,
\begin{equation*}
  \varphi_t\circ i_p=i_{\sigma(t)}\circ P_{\sigma,t}, \quad {\rm for} ~\forall~t\in (-\epsilon,\epsilon).
\end{equation*}
Hence we have
\begin{align*}
  i_p^*\mathcal{L}_XT=i_p^*\lim_{t\rightarrow0}\frac{\varphi_t^*T-T}{t}=\lim_{t\rightarrow0}\frac{P_{\sigma,t}^*i_{\sigma(t)}^*T-T_p}{t}
  =\lim_{t\rightarrow0}\frac{P_{\sigma,t}^*T_{\sigma(t)}-T_p}{t}.
\end{align*}
As $\mathcal{L}_XT\equiv 0~({\rm mod}~\omega^1,\ldots,\omega^n)$ if and only if $i_p^*\mathcal{L}_XT=0$ for any $p\in M$, the proof is complete.

\end{proof}

In \cite{Ichijyo78}, Ichijy\={o} proved that a Finsler manifold is a Landsberg manifold if and only if $\hat{g}$ is preserved by any parallel transports. This result can also be found in \cite{Bao,ChernShen}.
Following Lemma \ref{T}, Proposition \ref{horizontal lie g A} and \ref{horizontal lie eta}, we have the following proposition
in the same spirit.
\begin{prop}\label{further ich}
 Let $(M,\mathbf{F})$ be a Finsler manifold.  Then
  \begin{enumerate}
\item[(i)] $M$ is a Berwald manifold if and only if $\hat{A}$ is preserved by any parallel transports.
\item[(ii)] $M$ satisfies ${\rm tr}P=0$  if and only if $\eta$ is preserved by any parallel transports.
\end{enumerate}
\end{prop}
\begin{proof}[Proof of Theorem \ref{2}]
If $M$ is a Berwald manifold, Lemma \ref{T}, Proposition \ref{horizontal lie g A} and  \ref{horizontal lie eta} implies
that $\hat{g}$ and $\eta$ are preserved by any parallel transports. In the opposite direction, Theorem \ref{theo equivalence} and (\ref{radial linear of P}) implies
parallel transports are linear. Then Lemma \ref{lemma 2.2} implies $\hat{A}$ is preserved by all parallel transports. Hence (i) of Proposition \ref{further ich} shows $M$ is a Berwald manifold.
\end{proof}
We would like to list the following well known result as a corollary, even though the proof is indirect.
\begin{mcor}[\cite{Ichijyo76}]
 Let $(M,\mathbf{F})$ be a Finsler manifold. $M$ is a Berwald manifold if and only if all parallel transports are linear maps.
\end{mcor}
\begin{proof}
  If $M$ is a Berwald manifold, then $\hat{g}$ and $\eta$ are preserved by all parallel transports.
  By Theorem \ref{theo equivalence} and (\ref{radial linear of P}), all parallel transports are linear. In the opposite direction, if all parallel transports are linear,
  Lemma \ref{lemma 2.2} implies $\hat{A}$ is preserved. Then (i) of Proposition \ref{further ich} shows that $M$ is a Berwald manifold.
\end{proof}

\begin{proof}[Proof of Theorem \ref{3}]
\

(1)$\Rightarrow$(2): By Proposition \ref{prop eta=dtau}, $L=0$ and $S=0$ implies $\eta=d\tau$.
\par
(2)$\Rightarrow$(3): It is obvious.
\par
(3)$\Rightarrow$(4): By (\ref{deta}), $d\eta=0$ implies ${\rm tr}P=0$. $L=0$ and ${\rm tr}P=0$ implies $\hat{g}$ and $\eta$ are preserved by all parallel transports. By Theorem \ref{2}, $M$ is a Berwald manifold.
\par
 (4)$\Rightarrow$(1):  It is clear.
\end{proof}

\begin{proof}[Proof of Theorem \ref{theo L=0 and trB=0}]
 For Landsberg manifolds, the mean Berwald curvature coincides with the mean h-v curvature of Chern connection ${\rm tr}P$ (cf. \cite{BaoChernShen}, p. 67). Similar to (3)$\Rightarrow$(4) in Theorem \ref{3}, we complete the proof.
\end{proof}
\begin{remark}
  In Finsler geometry, some of concepts are \textit{priori} than linear connections.
  Spray, nonlinear connection, parallel transport, the three tensors $\hat{g}$, $\hat{A}$ and $\eta$ are such concepts.
  Berwald manifold and Landsberg manifold are often defined via linear connections.
  However,  Ichijy\={o}'s results point out these concepts are independent to linear connections.
  In contrast with some known results, for example Theorem 6.3 and 6.12 in \cite{SV}, the new characterizations
  of Berwald manifold in our paper do not involve any linear connection.
\end{remark}


\begin{thebibliography}{99}

\bibitem{Aikou10} Tadashi Aikou, \textsl{Some remarks on Berwald manifolds and Landsberg manifolds.}
Acta Math. Academiae Paedagogicae Ny\'{\i}regyh\'{a}ziensis, Vol. 26, 2010: 139-148.

\bibitem{Paiva}Juan Carlos \'{A}lvarez Paiva, \textsl{Some problems on Finsler geometry.} Handbook of Differential Geometry, Vol. 2, North-Holland, Elsevier, 2006: 1-33.




\bibitem{Bao} David Bao, \textsl{On two curvature-driven problems in Riemann-Finsler geometry.}
Advanced Studies in Pure Mathematics, Math. Soc. Japan, Vol.48, 2007: 19-71.

\bibitem{BaoChernShen} David Bao, Shiing-Shen Chern and Zhongmin Shen,
\textsl{An Introduction to Riemann-Finsler Geometry.} Graduate Texts
in Mathematics, Vol. 200, Springer-Verlag, New York, Inc., 2000.


\bibitem{Bla} Wilhelm Blaschke, \textsl{Vorlesungen \"{u}ber Differentialgeometrie II, Affine
Differentialgeometrie.} Springer, Berlin, 1923.



\bibitem{Bryant} Robert L. Bryant, \textsl{Some remarks on Finsler manifolds with constant flag curvature.} Houston J. Math, Vol. 28, 2002: 161-203.



\bibitem{Chern} Shiing-Shen Chern, \textsl{Local equivalence and Euclidean connections in Finsler spaces}. in Chern Selected papers, II, 1989: 95-121.

\bibitem{ChernShen} Shiing-Shen Chern and Zhongmin Shen, \textsl{Riemann-Finsler Geometry.} Nankai Tracts in Mathematics, Vol. 6, World Scientific, 2005.


\bibitem{FL} Huitao Feng and Ming Li, \textsl{Adiabatic limit and connections in Finsler geometry.}
Communications in Analysis and Geometry, Vol. 21, No. 3, 2013: 607-624.

\bibitem{H} Roland Hildebrand, \textsl{Centro-affine hypersurface immersions with parallel cubic form.} ArXiv:1208.1155v4, 2013.

\bibitem{Ichijyo76}Yoshihiro Ichijy\={o}, \textsl{Finsler manifolds modeled on a Minkowski space.} J. Math. Kyoto Univ., Vol. 16, 1976: 639-652.

\bibitem{Ichijyo78}Yoshihiro Ichijy\={o}, \textsl{On special Finsler connections with the vanishing hv-curvature tensor.} Tensor(N.S.), Vol. 32, 1978: 149-155.




\bibitem{Laugwitz0} Detlef Laugwitz,\textsl{Differentialgeometrie in Vektorr\"{a}umen.}  Friedr. Vieweg $\&$ Sohn, Braunschweig, 1965.

\bibitem{Laugwitz1} Detlef Laugwitz, \textsl{Zur Differentialgeometrie der Hyperfl\"{a}chen in Vektorr\"{a}umen und zur affingeometrischen Deutung der Theorie der Finsler-R\"{a}ume.}
Math. Z., Vol. 67, 1957: 63-74.

\bibitem{Laugwitz2} Detlef Laugwitz, \textsl{Eine Beziehung zwischen affiner und Minkowskischer Differentialgeometrie.}
Publ. Math. Debrecen, Vol. 5, 1957: 72-76.


\bibitem{LSZ}Anmin Li, Udo Simon and Guosong Zhao, \textsl{Global Affine Differential Geometry of Hypersurfaces.} W. de Gruyter, Berlin-New York, 1993.


















\bibitem{Mo} Xiaohuan Mo, \textsl{An Introduction to Finsler Geometry.} Peking University series in Math., Vol. 1, World Scientific Publishing Co. Pte. Ltd., 2006.


\bibitem{MoHuang} Xiaohuan Mo and Libing Huang, \textsl{On characterizations of Randers norms in a Minkowski space.}
International Journal of Mathematics, Vol. 21, No. 4, 2010: 523-535.



\bibitem{NS} Katsumi Nomizu and Takeshi Sasaki, \textsl{Affine Differential Geometry.} Cambridge University Press, 1994.



\bibitem{Schneider1} Rolf Schneider, \textsl{Zur affinen Differentialgeometrie im Gro{\ss}en. I.} Math. Zeitschr., Vol. 101, 1967: 375-406.

\bibitem{Schneider2} Rolf Schneider, \textsl{Zur affinen Differentialgeometrie im Gro{\ss}en. II.} Math. Zeitschr., Vol. 102, 1967: 1-8.



\bibitem{SSV} Udo Simon, Angela Schwenk-Schellschmidt and Helmut Viesel, \textsl{Introduction to the Affine Differential Geometry of Hypersurfaces.}
Lecture notes, Science University Tokyo, 1991. 




\bibitem{Shenbook1}Zhongmin Shen, \textsl{Differential Geometry of Spray and Finsler
Spaces.} Kluwer Acad. Publ., 2001.

\bibitem{Shenbook2}Zhongmin Shen, \textsl{Lectures on Finsler Geometry.} World Scientific, 2001.


\bibitem{Shen} Zhongmin Shen, \textsl{Landsberg curvature, S-curvature and Riemann curvature.} Riemann-Finsler Goemetry, MSRI Publications. Vol. 50, 2004: 303-355.







\bibitem{SV} J\'{o}zsef Szilasi and Csaba Vincze, \textsl{A new look at Finsler connections and special Finsler manifolds.} Acta Mathematica Academiae Paedagogicae Ny\'{\i} regyh\'{a}ziensis. New Series. Vol. 16, 2000: 33-63.


\bibitem{Zhang} Weiping Zhang, \textsl{Lectures on Chern-Weil Theory and Witten Deformations}. Nankai Tracts in Mathematics, Vol. 4, World Scientific Publishing Co. Pte. Ltd., 2001.


\end{thebibliography}
\end{document}